\documentclass{amsart}

\usepackage{amsmath, amssymb, amsthm, graphics, color}
\usepackage{amsaddr}

\title{Cubes and Their Centers}
\author{Riley Thornton}
\address{\small \emph{Department of Mathematics, University of California at Los Angeles \\ Los Angeles, CA, USA 90095\\
e-mail: personpants@math.ucla.edu}}
\thanks{This research was conducted as part of the Budapest Semesters in Mathematics Undergrauate Research Experience Program and was advised by Tam\'as Keleti.}


\newtheorem{thm}{Theorem}[section]
\newtheorem{lem}[thm]{Lemma}
\newtheorem{cor}[thm]{Corollary}

\theoremstyle{definition}

\theoremstyle{remark}

\newcommand{\R}{\mathbb{R}}
\newcommand{\N}{\mathbb{N}}
\newcommand{\Z}{\mathbb{Z}}
\newcommand{\Q}{\mathbb{Q}}

\newcommand{\inv}{^{-1}}

\newcommand{\setc}[2]{\mbox{\scalebox{.625}{${\left[\begin{array}{c} {#1} \\ {#2} \end{array}\right]}$}}}
\newcommand{\di}[1]{[[{#1}]]}

\begin{document}
\maketitle

\begin{abstract}
We study the relationship between the sizes of sets $B,S$ in $\R^n$ where $B$ contains the $k$-skeleton of an axes-parallel cube around each point in $S$, generalizing the results of Keleti, Nagy, and Shmerkin \cite{op} about such sets in the plane. We find sharp estimates for the possible packing and box-counting dimensions for $B$ and $S$. These estimates follow from related cardinality bounds for sets containing the discrete skeleta of cubes around a finite set of a given size. The Katona--Kruskal Theorem from hypergraph theory plays an important role. We also find partial results for the Haussdorff dimension and settle an analogous question for the dual polytope of the cube, the orthoplex.

\end{abstract}

\section{Introduction and Statements of Results} \label{sec:intro}

\subsection{Introduction}

In \cite{op} the authors find sharp bounds for the Hausdorff, box-counting, and packing dimensions of sets $S,B\subseteq \R^2$ where $B$ contains either the vertices or boundary of an axes-parallel square around every point in $S$, and cardinality bounds for finite sets satisfying discrete versions of these conditions. Their results are summarized in the following table. If $S$ has size $s$ for the given notion of size, then a sharp lower bound for the size of $B$ is given in terms of $s$:

\bgroup
\def\arraystretch{1.5}
\begin{center}
\begin{tabular}{c|c|c}

Notion of Size & Vertex problem & Boundary problem\\
 & ($0$-skeleton of a $2$-cube) & ($1$-skeleton of a $2$-cube) \\ \hline
$\dim_P$ & $\frac{3}{4}s$ & $1+\frac{3}{8}s$ \\
$\overline{\dim}_B$ & $\frac{3}{4}s$ & $\max\{1,\frac{7}{8}s\}$ \\
$\underline{\dim}_B$ & $\frac{3}{4}s$ & $\max\{1,\frac{7}{8}s\}$ \\
$\dim_H$ & $\max\{0,s-1\}$ & 1 \\
$|\cdot |$ & $\Omega(s^{\frac{3}{4}})$ & $\Omega(s^{\frac{7}{8}-\epsilon})$ 

\end{tabular}
\end{center}
\egroup

%STRETCH TABLE

We want to find similar bounds for sets $B,S\subseteq \R^n$ where $B$ contains the $k$-skeleton of a $n$-cube around each point in $S$. That is, the goal of this paper is to extend the above table into an $\N\times \N$ array. We manage this in every case except for the Hausdorff dimension.

The authors in \cite{op} were inspired by Bourgain \cite{Bourgain} and Marstrand's \cite{Marstrand} results about packing circles, which in turn completed work done by E. Stein work on $n$-spheres for $n\geq 3$ \cite{nballs} The problem we study, the original problem in \cite{op}, and the problems settled by Stein and (independently) Bourgain and Marstrand are members of the abundant family of ``Kakeya type" problems. The problem of minimizing size, in the sense of some measure of fractal dimension, over some family of sets in $\R^d$ occurs commonly in geometric measure theory and other areas of analysis. The analagous discrete problem of optimizing cardinality with respect to some combinatorial constraint is the domain of extremal combinatorics. This paper and \cite{op} illuminate some relationships between these two areas of research.

\subsection{Notation}

Throughout, lowercase latinate letter stand for integers unless otherwise specified. The expression $\di{a,b}$ stands for the discrete interval $\{a,a+1,...,b-1,b\}$ and $\setc {n}{k}$ for $k$ element subsets of $\di{1,n}$. For $x\in \R^n, I\in \setc n k$, $x_I$ is the vector in $\R^k$ formed by taking the entries of $x$ indexed by $I$. For any finite set $A$, $|A|$ is the number of elements of $A$.

When we consider sets $X\subseteq \R^m$ and $Y\subseteq \R^n$, and functions $f,g$ possibly dependent on $X,Y$, we will say $f\leq O(g)$ to mean that $f\leq Cg$, and $g\geq\Omega(f)$ to mean $g\geq cf$ for some constants $C,c\in \R^+$ depending only on $m,$ and $n$. Note that $f\leq O(g)$ iff $g\geq\Omega(f)$.

We use the convention that $0\not\in\N.$ Unless otherwise specificed, a cube will mean a cube with all sides parallel to the axes, that is, a set of the form $x+[a,b]^n$ for some $x\in\R^n$ and $a<b\in \R$. The $k$-skeleton of the cube $x=[a,b]^n$ is the set $x+\cup_{I\in\setc{n}{k}}\prod_{i=1}^{n} A_{i,I}$ where $A_{i,I}$ is $[a,b]$ if $i\in I$ and $\{a,b\}$ otherwise. A discrete cube is a set of the form $x+\di{0,r}^n$ for some $x\in \Z^n$ and $r\in\N$. The discrete $k$-skeleton of the discrete cube $x+\di{0,r}^n$ is then $x+\cup_{I\in\setc{n}{k}} \prod_{i=1}^n A_{i,I}$ where $A_{i,I}$ is $\di{0,r}$ if $i\in I$ and $\{0,r\}$ otherwise.

The notions of dimensions we consider are the Hausdorff dimension, $\dim_H$, the upper and lower box-counting dimension, $\overline{\dim}_B$ and $\underline{\dim}_B$ respectively, and the packing dimension $\dim_P$. When upper and lower box-counting dimension agree for a set, we just write $\dim_B$. See \cite{packing} for definitions and details.

\subsection{Main Results}

The main results of the paper are the generalizations of the bounds for box-counting and packing dimensions in \cite{op}:

\begin{thm}\label{int:pbbounds} For any $0\leq k<n$ and any sets $B,S\subseteq \R^n$ such that $B$ contains the $k$-skeleton of a cube around every point in $S$,
\begin{enumerate}
\item $\dim_P(B)\geq k+\frac{(n-k)(2n-1)}{2n^2}dim_P(S)$, and
\item $\overline{\dim}_B(B)\geq \max\left\{k,\left(\frac{k}{n}+\frac{(n-k)(2n-1)}{2n^2}\right)\overline{\dim}_B(S)\right\}
=\max\left\{k,\left(1-\frac{n-k}{2n^2}\right)\overline{\dim}_B(S)\right\}$, and similarly $\underline{\dim}_B(B)\geq\max\left\{k,\left(1-\frac{n-k}{2n^2}\right)\underline{\dim}_B(S)\right\}$.
\end{enumerate}
\end{thm}
We also have constructions showing that these bounds are sharp in the sense of the following theorem:

\begin{thm}\label{int:pbcon}
Given any $0\leq k<n$, $s\in[0,n]$, there are compact sets $B,S, B', S'\subseteq\R^n$ with $\dim_P(S)=\dim_B(S')=s$, where $B$ and $B'$ contain the $k$-skeleton of a cube around every point in $S$ and $S'$ respectively, and
\begin{enumerate}
\item $\dim_P(B)=k+\frac{(n-k)(2n-1)}{2n^2}s$, and
\item $\dim_B(B')=\max\left\{\left(1-\frac{n-k}{2n^2}\right)s, k\right\}$.
\end{enumerate}
\end{thm}

The above results are analytic extensions of the following four theorems for discrete cubes:

\begin{thm}\label{int:discbound}
If $B,S\subseteq\Z^n$ are finite, and $B$ contains the discrete $k$-skeleton of a cube around every point in $S$, then for every $\alpha<1-\frac{n-k}{2n^2}$,
\[|B|\geq \Omega(|S|^\alpha).\]
\end{thm}

\begin{thm}\label{int:nllem}
If $0<\ell\leq n$, $A\subseteq\R^\ell$, $S\subseteq \R^n$ are finite sets such that
\begin{equation} \label{cond2.4box}\forall x\in S\;\exists r\in\R^+\;\forall I\in \setc n \ell\;\forall \sigma\in \{-1,1\}^\ell:\;\; x_I+r\sigma\in A,
\end{equation}
then $|A|\geq \Omega(|S|^{\ell(2n-1)/2n^2}).$
\end{thm}
Geometrically (\ref{cond2.4box}) above means that $A$ contains the vertices of an $\ell$-cube around each point in every projection of $S$ onto an axis-spanned plane. Intuitively, $A$ collects the non-$(n-k)$-face cofactors of cubes around each point in $S$. We also have constructions showing these are sharp:

%The sentences starting Intuitively: k? is this intuitive? what the hell am I saying?

\begin{thm}\label{int:disccon}
For every  $0\leq k<n$ and every $p\geq 0$, there are finite $B,S\subseteq \Z^n$ such that $B$ contains the discrete $k$-skeleton of a cube around every point in $S$, $|S|=p$, and $|B|\leq O\left(|S|^{1-\frac{n-k}{2n^2}}\right)$.
\end{thm}

\begin{thm}\label{int:nlcon}
For every $0<\ell\leq n$ and every $p\geq 0$, there are finite $B\subseteq \Z^\ell$, $S\subseteq \Z^n$ such that $|B|\le O(|S|^{\ell(2n-1)/(2n^2)})$, $|S|=p$, and
\[\forall x\in S\;\exists r\in\R^+\;\forall I\in \setc n \ell\;\forall \sigma\in \{-1,1\}^\ell:\;\; x_I+r\sigma\in B.\]

\end{thm}

We have a general bound for the Hausdorff dimension and have shown its sharpness in several cases.

\begin{thm}\label{int:hausbound}
If $B,S\subseteq\R^n$ and $B$ contains the $k$-skeleton of a cube around every point in $S$, then $\dim_H(B)\geq \max\{\dim_H(S)-1,k\}.$ 
\end{thm}

\begin{thm}\label{int:hauscon}
For $0\leq k<n$, $s\in [n-k,n]$, there are $G_{\delta}$ sets $B,S\subseteq \R^n$ where $B$ contains the $k$-skeleton of an $n$-cube around each point in $S$, $\dim_H(B)=\max\{k,s-1\}$ and $\dim_H(S)=s$.

Further, if $k=0$, there are $B,S$ as above for $s\in[n-1,n]$.
\end{thm}

We conjecture that constructions as in the theorem above can be found for all $s$. Marianna Cs\"ornyei and Tam\'as Keleti have recently proven this conjecture \cite{CK}. These results are summarized in the following table:

\begin{center}
\begin{tabular}{c|c}

Notion of Size & $k$-skeleton of an $n$-cube \\ \hline
 &  \\
$\dim_P$ & $k+\frac{(n-k)(2n-1)}{2n^2}\dim_P(S)$\\
$\overline{\dim}_B$ & $\max\{k,(1-\frac{n-k}{2n^2})\overline{\dim}_B(S)\}$ \\
$\underline{\dim}_B$ & $\max\{k,(1-\frac{n-k}{2n^2})\underline{\dim}_B(S)\}$ \\
$\dim_H$ & $\max\{k,\dim_H(S)-1\}$ (conjectured)  \\
$|\cdot |$ & $O\left(|S|^{1-\frac{n-k}{2n^2}}\right)$ 

\end{tabular}
\end{center}

We also have bounds for sets $B$ in $\R^n$ containing the vertices of the dual polytope of the cube, the orthoplex, around every point in a set $S$ of a given dimension:

\begin{thm}\label{int:orthobound}
Let $B,S\subseteq\R^n$ such that for all $x\in S$ there is some $r\in\R^+$ such that $x\pm re_i\in B$, where $e_i$ is the $i^{th}$ standard basis vector. Then, the following hold:

\begin{enumerate}
\item $\dim_H(B)\geq\dim_H(S)-1$,
\item if $B,S$ are finite, then $|B|\geq\Omega(|S|^{\frac{2n-1}{2n}})$,
\item $\underline{\dim}_B(B)\geq \frac{2n-1}{2n}\underline{\dim}_B(S)$ and $\overline{\dim}_B(B)\geq \frac{2n-1}{2n}\overline{\dim}_B(S)$, and
\item $\dim_P(B)\geq\frac{2n-1}{2n}\dim_P(S)$.
\end{enumerate}
\end{thm}

And, these bounds are sharp:
\begin{thm}\label{int:orthocon}
For any $n, p\in \N$ and $s\in[0,n]$, we can find compact sets $B_H$, $B_B$, $B_P$, $B_f$, $S_H$, $S_B$, $S_P,$ and $S_f\subseteq \R^n$ such that $B_X$ contains the vertices of an orthoplex around each point in $S_X$ for any $X\in\{H,B,P, f\}$, $\dim_X(S_X)=s$ for $X\in\{H,B,P\}$, $|S_f|=p$, and
\begin{enumerate}
\item $\dim_H(B_H)= \max\{1,\dim_H(S_H)-1\}$
\item $|B_f|= O(|S_f|^{\frac{2n-1}{2n}})$ (in particular, $S_f, B_f$ are finite)
\item $\overline{\dim}_B(B_B)= \frac{2n-1}{2n} \underline{\dim}_B(S_B)$
\item $\dim_P(B_P)= \frac{2n-1}{2n}\dim_P(S_P)$.
\end{enumerate}
\end{thm}

\subsection{Structure of the Paper}
The rest of this paper is organized as follows. In section \ref{sec:disc}, we prove the discrete results. In section \ref{sec:dim} we collect results in dimension theory which we will require. In section \ref{sec:estimates}, we prove bounds for box-counting and packing dimension. In section \ref{sec:packcon} we give constructions showing the sharpness of the bounds for packing dimension and of the bound for box-counting dimension in the case $k=0$ (the vertex case). In section \ref{sec:boxcon} we give a construction showing the sharpness of the box-counting bound. In section \ref{sec:haus} we cover what is known about Hausdorff dimension and pose a conjecture about this case. In section \ref{sec:ortho} we settle the vertex problem for the orthoplex.

\section{Discrete Results} \label{sec:disc}

The sets we construct and bound in this section will be finite. Most of the discrete results follow as corollaries from two lemmas: a construction based on $i$-ary expansions and a bound relating sets in dimension $n$ and $\ell$ for any $\ell<n$. We give the construction first. The lemma below generalizes \cite[Lemma 4.3]{op}.

\begin{lem}[Digit Construction] \label{digit}
For every $n, i$, there is a $D_{i,n}\subseteq \di{-(2i)^{2n},(2i)^{2n}}$ of size $O(i^{2n-1})$ such that, for every $x_1,...,x_n\in \di{1,i^{2n}-1}$, there is some $r$ such that, for every $0< j\leq n$, $x_j\pm r\in D_{i,n}$ and $|r|\in \di{1, i^{2n}}$.
\end{lem}

\begin{proof}

The discrete interval $\di{1,i^{2n}-1}$ is the set of integers which can be written in base $i$ with $2n$ digits and at least 1 nonzero digit. Informally, $D_{i,n}$ is the set of numbers with at least one $0$ in their base $i$ expansion if we allow for negative digits.

\[D_{i,n}=\left\{\sum_{j=0}^{2n-1} a_{j}i^j: a_j\in\di{2(1-i),2(i-1)}, \prod_j a_j=0\right\}.\]

 Note that $|D_{i,n}|\leq (4i-4)^{2n}-(4i-5)^{2n}=O(i^{2n-1})$. Given $n$ numbers $x_1,...,x_n\in \di{1,i^{2n}-1}$, denote the terms of their $i$-expansions by

\[x_j=x_{j,2n-1}...x_{j,0},\]

that is
\[
x_j=\sum_{m=0}^{2n-1} x_{j,m}i^m\quad\mbox{with } 0\leq x_{j,m}<i,
\]

and let $r=x_{n-1,2n-2}0...0x_{1,2}0x_{0,0}-x_{n-1,2n-1}...x_{1,3}0x_{0,1}0$, that is

\[
r=\left(\sum_{m=0}^{n-1} x_{m,2m}i^{2m}\right)-\left(\sum_{m=0}^{n-1} x_{m,2m+1}i^{2m+1}\right). 
\]

%R may not be positive. To fix this, we can choose another r, but we will need to exclude x_i=0

By permuting the $x_i$, we may assume that at least one of $x_{i,2i}$ or $x_{i,2i+1}$ is nonzero, and so $r$ is nonzero. For any $j$, 
\[x_j+r=\left(\sum_{m=0}^{n-1} (x_{j,2m}+x_{m,2m})i^{2m}\right)+\left(\sum_{m=0}^{n-1} (x_{j,2m+1}-x_{m,2m+1})i^{2m+1}\right).\] We have that $|x_{j,m}\pm x_{\lfloor m/2\rfloor,m}|\le |x_{j,m}|+|x_{\lfloor m/2\rfloor,m}|\le 2(i-1)$, and $x_{j,2j+1}-x_{j,2j+1}=0$, so $x_j+r\in D_{i,n}$. And similarly for $x_j-r$. 
\end{proof}

\begin{cor}[Theorem \ref{int:nlcon} in the Introduction] \label{nlcon}
For every $n\geq \ell>0$ and $p\geq 0$, there are $B\subseteq \Z^\ell$, $S\subseteq \Z^n$ such that $|B|\le O(|S|^{\ell(2n-1)/(2n^2)})$, $|S|=p$, and
\[\forall x\in S\;\exists r\in\R^+\;\forall I\in \setc n \ell\;\forall \sigma\in \{-1,1\}^\ell:\;\; x_I+r\sigma\in B.\]
\end{cor}

\begin{proof}
If $p=i^{2n^2}$ for some $i$, take $B=(D_{i,n})^\ell$ and $S=\di{1,i^{2n}-1}^n$. For intermediate values of $p$, interpolate by taking $i$ to be the smallest such that $p\leq i^{2n^2}$ and let $S$ be a subset of $\di{1,i^{2n}-1}^n$ of size $p$, and $B=(D_{i,n})^\ell$. Since $(i+1)^{2n^2}-i^{2n^2}\leq 2^{n^2}i^{2n}$, $i^{2n^2}\leq O(p)$. The desired inequality follows directly.
\end{proof}

%clean up

\begin{cor}[Theorem \ref{int:disccon} in the Introduction]\label{disccon}
For every $n>k\geq 0$ and $p\geq 0$, there are $B,S\subseteq \Z^n$ such that $B$ contains the discrete $k$-skeleton of a cube around every point in $S$, $|S|=p$, and $|B|\leq O\left(|S|^{1-\frac{n-k}{2n^2}}\right)$.
\end{cor}

\begin{proof}
If $p=i^{2n^2}$, take $B=\bigcup_{J\in \setc n k} \prod_{j=1}^{n} A_j$, where
\[A_j=\left\{\begin{array}{ll} D_{i,n} & j\not\in J \\ \; \di{1,i^{2n}-1} & j\in J \end{array}\right.,\]
$S=\di{1,i^{2n}-1}^n$. Otherwise, interpolate as above.
\end{proof}

To get bounds showing that the above constructions are optimal, we will need a theorem comparing sets in $\R^\ell$ to sets in $\R^n$, the $(n,\ell)$-dimensional lemma, Lemma \ref{nllem}. To get this, we start by proving the case where $n=\ell$, the $n$-dimensional lemma, Lemma \ref{nlem}. This is a generalization of \cite[Theorem 4.1]{op}. The idea of the proof there is to decompose the square into two intersecting lines and use a bound in $\R$ to get a bound in $\R^2$. Our argument similarly decomposes the cube into a line and an $(n-1)$-plane and proceeds induction.

We make a slightly more opaque statement of the lemma because this will be useful for resolving questions about the orthoplex and because the proof is more natural.

\begin{lem}[$n$-dimensional Lemma] \label{nlem}
For any $n>0$, $B,S \subseteq \R^n$ finite, if there are $v_1,...,v_n$ linearly independent vectors such that for any $x\in S$ and $0\leq i\leq n$, there is some $r\in\R^+$ such that $x\pm r v_i\in B$ (and in particular, if $B$ contains the vertices of an $n$-cube around each point in $S$), then $|B|\geq \Omega\left(|S|^{(2n-1)/(2n)}\right)$.

\end{lem}
\begin{proof}
We will prove that, in fact, $|B|\geq \frac{1}{2^{n-1}}|S|^{(2n-1)/(2n)}$. We proceed by induction.

If $n=1$, then any point in $S$ is the midpoint of two points in $B$, so $|S|\leq {|B|\choose 2}\leq |B|^2$. 

So, suppose that the bound holds for $\R^n$; we want to verify it for $\R^{n+1}$. Let $c=\frac{1}{2^n}$. Consider the planes $P_1,...,P_k$ and lines $\ell_1,...,\ell_m$ through points in $S$ with each $P_i$ parallel to $\mathrm{span}(v_1,...,v_{n})$ and $\ell_i$ parallel to $v_{n+1}$. Let $p_i=|S\cap P_i|$ and $q_i=|S\cap \ell_i|$.

Note that for each $s\in P_i\cap S$, there is some $r$ such that $P_i\cap S$ contains $s\pm rv_i$ for $1\leq i\leq n$. So, by the inductive hypothesis:

\[|B\cap P_i|\geq 2c p_i^{(2n-1)/(2n)}\quad\mbox{and}\quad|B\cap \ell_i|\geq q_i^{1/2},\]

so that
\[|B|\geq 2c \sum_{i\leq k} p_i^{(2n-1)/(2n)}\quad\mbox{and}\quad |B|\geq\sum_{i\leq m} q_i^{1/2}.\]
Let $a_1,...,a_w$ be the $p_i$ which are less than $|S|^{n/(n+1)}$. There are two cases we need to consider.

\paragraph{Case 1: $a_1+...+a_w\geq\dfrac{1}{2}|S|$.} Here we have the following (using $a_i\leq |S|^{\frac{n}{n+1}}$ for the third inequality):

\[|B|\geq 2c\sum_{i\leq k} p_i^{\frac{2n-1}{2n}}
\geq 2c\sum_{i\leq w} a_i^{\frac{2n-1}{2n}}
\geq 2c \sum_{i\leq w} \dfrac{a_i}{|S|^{\frac{1}{2(n+1)}}}
\geq c|S|^{\frac{2(n+1)-1}{2(n+1)}}.\]

\paragraph{Case 2: $a_1+...+a_w<\dfrac{1}{2}|S|.$}

Here, for each plane $P_i$ where $p_i>|S|^{\frac{n}{n+1}}$ color $P_i\cap S$ blue. Note that there are at most $|S|^{\frac{1}{n+1}}$ planes with blue points. Let $q'_i$ be the number of blue points on $\ell_i$. Since the $\ell_i$ do not lie in any $P_i$ we have

\[q'_i\leq |S|^{\frac{1}{n+1}} \quad\mbox{and}\quad q'_i\leq q_i.\]

Also, since every point in $S$ has some line $\ell_i$ going through it

\[\sum_{i\leq m} q_i' \geq |S|-a_1-...-a_w\geq \dfrac{1}{2}|S|,\]

so that

\[|B|\geq \sum_{i\leq m} \sqrt{q_i'} 
\geq \sum_{i\leq m} \dfrac{q_i'}{|S|^{\frac{1}{2(n+1)}}}
\geq \dfrac{1}{2} |S|^{\frac{2(n+1)-1}{2(n+1)}}
\geq c |S|^{\frac{2(n+1)-1}{2(n+1)}}.\] 
\end{proof}

To deduce the $(n,\ell)$-lemma, the generalization of \cite[Theorem 4.2]{op}, we will appeal to Lovasz's corrolary of the Katona--Kruskal Theorem:

\begin{thm}[Katona \cite{Ka}, Kruskal \cite{Kr}]
Let $Y$ be a set of $b$-element sets for some $b\in\N$, let $X$ be the $(b-c)$ element subsets of sets in $Y$, and let
\[|Y|={n_1 \choose k}+{n_2\choose k-1}+...+{n_j\choose k-j},\]
where the $n_i$ are a sequence of nonnegative decreasing integers. Then,
\[|X|\geq {n_1 \choose k-c}+{n_2\choose k-c-1}+...+{n_j\choose k-c-j}.\]
\end{thm}

\begin{cor}[Lovasz \cite{L}]
Let $X,Y, b,c$ be as above, and let $x\in \R$ be such that $|Y|={x\choose b}$, then $|X|\geq {x \choose b-c}$, where ${x \choose b}=\frac{x(x-1)...(x-b+1)}{b!}$.
\end{cor} 

A short proof of these is given in \cite{kklproof}.

\begin{thm}[$(n,\ell)$-dimensional Lemma]\label{nllem}
If $A\subseteq \R^\ell$, $S\subseteq \R^n$ are finite, and
\begin{equation} \label{nlcond} \forall x\in S\;\exists r\in\R^+\;\forall I\in \setc n \ell\;\forall \sigma\in \{-1,1\}^\ell:\;\; x_I+r\sigma\in A, \end{equation}
then $|A|\geq \Omega\left(|S|^{\ell(2n-1)/(2n^2)}\right).$
\end{thm}

\begin{proof}

%explain:
The condition (\ref{nlcond}) implies that $B=\{x\in \R^n: \forall I\in \setc n \ell \;x_I\in A\}$ and $S$ satisfy the hypotheses of the $n$-dimensional lemma, Lemma \ref{nlem}; indeed, let $x\in S$, then by linearity of projections

\[\exists r\in\R^+\;\forall\sigma\in\{-1,1\}^n\;\forall I\in\setc n \ell :\quad x_I+\sigma_Ir\in A \quad\iff\]
\[\exists r\in\R^+\;\forall\; \sigma\in\{-1,1\}^n\; \forall I:\quad (x+\sigma r)_I\in A \quad\iff\]
\[\exists r\in\R^+\;\forall \sigma\in\{-1,1\}^n:\quad x+\sigma r\in B.\]
So, by the previous lemma we have

\begin{equation} \label{nlb} |B|\geq \Omega(|S|^{(2n-1)/(2n)}) .\end{equation}

 To compare $B$ and $A$, we can make the following simplifications. By translating $A$ appropriately we may assume that for any $x\in A$, the $x_i$ are distinct. Let the symmetric group $\Sigma_\ell$ act on $\R^m$ by $\pi(x_1,...,x_\ell)=(x_{\pi(1)},...,x_{\pi(\ell)})$. Taking the orbit of $A$ under $\Sigma_\ell$ only increases the size $A$ by a factor of $\ell!$, so we may assume that if $x\in A$, then $(x_{\pi 1},...,x_{\pi \ell})\in A$ for any $\pi\in \Sigma_\ell$. Now let \[\hat A=\{\{x_1,...,x_\ell\}: x\in A\}\quad \hat B=\{\{x_1,...,x_n\}: x\in B\}.\]

Note that $|A|\geq |\hat A|$ and $|B|\leq n!|\hat B|$. If $C\in \hat B$ and $C'\subset C$ with $|C'|=\ell$, then $C'\in \hat A$, so we may apply the Katona--Kruskal--Lovasz Theorem to $\hat B$ and $\hat A$. Let $x\in\R$ be such that ${{x}\choose {n}}=|\hat B|$. Then $|\hat A|\geq {x\choose \ell}$. So, we have then for some $c\in \R^+$:

\[|A|\geq |\hat A|= {x \choose \ell }\geq c x^\ell\]
\[|B|\leq n!|\hat{B}|=n!{x \choose n}\leq x^n.\]

Combining this with (\ref{nlb}), we get \[|A|\geq \Omega\left(|B|^{\frac{\ell}{n}}\right)\geq \Omega\left(|S|^\frac{\ell(2n-1)}{2n^2}\right).\]

\end{proof}

We will see that the packing dimension estimate reduces to almost exactly the above theorem. The next theorem finishes off the discrete problem and will later be used to give a bound for box-counting dimension. The argument given below generalizes an unpublished proof by D\'aniel T. Nagy \cite{N}.

\begin{thm}[Theorem \ref{int:discbound} in the Introduction]\label{discbound}
If $B, S\subseteq \Z^n$ are finite, and $B$ contains the discrete $k$-skeleton of a discrete cube around each point in $S$, then, for every $\alpha<1-\frac{n-k}{2n^2}$,
\[|B|\geq \Omega(|S|^\alpha).\]
\end{thm}

\begin{proof}
Let $R(\alpha)=\frac{2n^2-(2n-1)(n-k)}{(2n^2)(k+n(1-\alpha))}$, $f(\alpha)=R(\alpha)k+\frac{(2n-1)(n-k)}{2n^2}$, and $\beta=1-\frac{n-k}{2n^2}$. Call $\alpha$ \emph{good} if, whenever $B$ contains the discrete $k$-skeleton of a cube around each point in $S$, then $|B|\geq \Omega(|S|^\alpha)$.  Note that the discrete construction, corollary \ref{disccon}, shows that any good $\alpha$ is at most $\beta.$

 We will show that if $\alpha$ is good, so is $f(\alpha)$ and that $\lim_{n}f^n(0)=\beta$. One can check that $R(\alpha)$ has been chosen so that 

\[f(\alpha)= 1-(R(\alpha)n(1-\alpha)).\]

Now, let $\alpha$ be good. Given sets $S,B$ as in the statement, call a cube in $B$ \emph{large} if it has side length at least $|S|^{R(\alpha)}$, and \emph{small} otherwise.

\paragraph{Case $1$:} Suppose at least $\frac{|S|}{2}$ points of $S$ are centers of large cubes. For any set $X\subseteq \R^n$, $1\leq \ell\leq n$,  and $I\in\setc n \ell$, let $X_I=\{x_I:x\in X\}$. Let $V$ be the of vertices of large cubes in $B$ and $A=\bigcup_{I\in \setc n {n-k}} V_I$. Let $S_1$ be the set of centers of the large cubes. Then, $A,S_1$ satisfy the $(n,n-k)$-dimensional lemma (Lemma \ref{nllem} with $l=n-k$), so

\[|A|\geq \Omega\left(|S_1|^{\frac{(n-k)(2n-1)}{2n^2}}\right)\geq\Omega\left(|S|^{\frac{(n-k)(2n-1)}{2n^2}}\right).\]

At least one $V_I$ will contain ${n\choose k}\inv|A|\geq\Omega\left(|S|^{\frac{(2n-1)(n-k)}{2n^2}}\right)$ points. The set $B$ contains $\Omega(|V_I|)$ $k$-faces of large cubes whence

%Really think on this

\[|B|\geq \Omega(|V_I||S|^{kR(\alpha)})\geq \Omega\left(|S|^{R(\alpha)k+\frac{(2n-1)(n-k)}{2n^2}}\right)=\Omega(|S|^{f(\alpha)}).\]

\paragraph{Case $2$:} Suppose that $\frac{|S|}{2}$ points of $|S|$ are centers of small cubes in $B$. Denote these centers by $S_2$. Divide $\Z^n$ into cubes of side length $|S|^{R(\alpha)}$. Assume that these cubes of the partition contain $x_1,...,x_m$ points of $S_2$ (ignoring the empty cubes). Then, 

\[1\leq x_i\leq |S|^{nR(\alpha)},\quad \sum x_i=|S_2|\geq \frac{|S|}{2}.\]

For each point $b$ of $B$ consider the small cubes containing $b$. The centers of these $n$-cubes cannot be in more than $2^n$ cubes from the partition. Let $Y_i$ be the union of small cubes in $B$ with centers counted in $x_i$. We have $\sum_i |Y_i|\leq2^n |B|$, and $|Y_i|\geq \Omega(|x_i^{\alpha}|)$ since $\alpha$ is good. So,
\begin{align*} 
2^n|B| & \geq \sum \Omega(x_i^\alpha) \\
	 & = \sum \Omega\left(\frac{ x_i}{x_i^{(1-\alpha)}} \right)\\
            &  \geq \Omega\left(\frac{\sum x_i}{|S|^{R(\alpha)n(1-\alpha)}}\right) \\
            & =\Omega\left(|S|^{1-(R(\alpha)n(1-\alpha))}\right) \\
            & =\Omega\left(|S|^{f(\alpha)}\right).
\end{align*}

Algebra shows that $f$ has two fixpoints, $1$ and $\beta$. By inspection $f$ is monotone increasing on the interval $[0, \beta]\subset [0,1]$. We then have

\[0\leq f(0)\Rightarrow f^n(0)\leq f^{n+1}(0)\]
and
\[0<\beta\Rightarrow f^n(0)\leq f^n(\beta)=\beta.\]

So, the sequence $f^n(0)$ must converge to $\beta$.
\end{proof}

\section{Dimension Theory Primer}\label{sec:dim}

Before proving results about dimension, we collect results from the general theory that we will use. First, we can equivalently characterize packing dimension as the countably stabilized box-counting dimension.

\begin{thm}[Packing Dimension Equivalents] \label{dimequiv}
The packing dimension and modified box-counting dimension are equivalent. That is, for $A\subseteq \R^n$

\[\dim_P(A)=\overline{\dim}_{MB}(A)=\inf\left\{\sup_{i\in\N} \overline{\dim}_B(A_i): A\subseteq \bigcup_{i\in\N} A_i\right\}.\]
\end{thm}

See, for instance, \cite[Proposition 3.8]{packing} for a proof. Note that, since box-counting dimension is finitely stable (see, for instance \cite[Section 3.2]{packing}), we can require ascending unions:
\[\dim_P(A)=\inf\left\{\sup_{i\in\N} \overline{\dim}_B(A_i): A\subseteq \bigcup_{i\in\N} A_i, \; A_i\subseteq A_{i+1}\right\}.\]

It is helpful to have comparisons between the different notions of dimension.

\begin{thm}[Dimension Inequalities] \label{dimineq}
For $A\subseteq \R^n$ and $B\subseteq\R^\ell$, the following inequalities hold
\[\dim_H(A)\leq \dim_P(A),\quad\underline{\dim}_B(A)\leq \overline{\dim}_B(A).\]

\end{thm}

In general, $\dim_P(A)$ and $\underline{\dim}_B(A)$ are not comparable. For proof, see e.g. \cite[Theorem 8.10]{inequalities}. Finally, we will want to compare the dimensions of products of sets.

\begin{thm}[Product Rules] \label{dimprod}
For $A\subseteq\R^n$ and $B\subseteq \R^\ell$, the following inequalities hold:

\begin{enumerate}

\item $\dim_H(A)+\dim_H(B)\leq\dim_H(A\times B)\leq\dim_H(A)+\overline{\dim}_B(B)$,

\item $\dim_P(A)+\dim_H(B)\leq\dim_P(A\times B)\leq \dim_P(A)+\dim_P(B)$,

\item $\underline{\dim}_B(A\times B)\geq \underline{\dim}_B(A)+\underline{\dim}_B(B)$, and

\item $\overline{\dim}_B(A\times B)\leq \overline{\dim}_B(A)+\overline{\dim}_B(B).$

\end{enumerate}
\end{thm}

%This isn't actually a correct citation

See \cite{dimprodref} for proofs of $(1)$ and $(2)$ and \cite{sharples} for $(3)$ and $(4)$. All of the above inequalities in this section be strict.

\section{Packing and Box-Counting Estimates}\label{sec:estimates}

The $(n,\ell)$-dimensional lemma, and discrete bounds give rise to dimensional analogues. The proofs sketched in \cite{op} apply almost directly. For completeness, we fill in some details below.

\begin{thm}[Theorem \ref{int:pbbounds} part 2 in the Introduction] \label{boxbound}
If $B,S\subseteq \R^n$ and $B$ contains the $k$-skeleton of an $n$-cube around every point in $S$, then
\[\overline{\dim}_B(B)\geq \max\left\{\left(\frac{k}{n}+\frac{(n-k)(2n-1)}{2n^2} \right)\overline{\dim}_B(S),k\right\}\]
and similarly
\[\underline{\dim}_B(B)\geq \max\left\{\left(\frac{k}{n}+\frac{(n-k)(2n-1)}{2n^2} \right)\underline{\dim}_B(S),k\right\}.\]
\end{thm}
\begin{proof}
Since $B$ contains a $k$-cube, $\overline{\dim}_B(B)\geq k$.

For $x\in \R^n$ and $r\in \R$, let $\widetilde x_m$ and $\tilde r$ be the centers of the half-open dyadic cube and interval of side length $2^{-m}$ containing $x$ and $r$ respectively. Define $\widetilde S_m:=\{\widetilde x_m: x\in S\}$. Without loss of generality, $B$ is the union of $k$-skeletons of cubes with centers in $S$. That is, $B=\bigcup_{x\in S} C(x, r(x))$, where $C(x,r(x))$ is the $k$-skeleton of the cube centered at $x$ with side length $2r(x)$ depending on $x$. Let $\widetilde B_m=\bigcup_{x\in S} C(\widetilde x_m, \widetilde{r(x)}_m)$. It is clear that $B$ meets $\Omega(|\widetilde B_m|)$ cubes in $2^{-m}\Z$, and from the discrete bound, Theorem \ref{discbound},
\[|\widetilde B_m|\geq\Omega(|\widetilde S_m|^{1-\frac{n-k}{2n^2}}).\]

So, from some $c_1,c_2$ depending only on $n,k$,
\begin{align*}\underline{\dim}_B(B)&\geq \liminf_m \frac{\log(|\widetilde B_m|)+c_1}{m}\\
\;&\geq \liminf_m \frac{\log(|\widetilde S_m|)\left(1-\frac{n-k}{2n^2}\right)+c_2}{m}\geq\underline{\dim}_B(S)\left(1-\frac{n-k}{2n^2}\right),\end{align*}
and similarly for $\overline{\dim_B}$.
\end{proof}

\begin{lem} \label{nlbox} If $A\subseteq\R^\ell$, $S\subseteq \R^n$, and 
\[\forall x\in S\;\exists r\in \R \;\forall I\in \setc n \ell\;\forall \sigma\in\{-1,1\}^\ell:\quad x_I +\sigma r\in A\]

then, 
\[\underline{\dim}_B(A)\geq \left(\frac{\ell(2n-1)}{2n^2}\right)\underline{\dim}_B(S)\]
and
\[\overline{\dim}_B(A)\geq \left(\frac{\ell(2n-1)}{2n^2}\right)\overline{\dim}_B(S)\]

\end{lem}
\begin{proof}
The proof is exactly as above, except appealing to the $(n,\ell)$-dimensional lemma, Lemma \ref{nllem} instead of Theorem \ref{discbound}.
\end{proof}

\begin{lem} \label{nldim}

If $A\subset \R^\ell$, $S\subset \R^n$, and
\[\forall x\in S\;\forall I\in \setc n \ell \;\exists r\in \R \;\forall \sigma\in\{-1,1\}^\ell:\quad x_I +\sigma r\in A\]
then
\[\dim_P(A)\geq \left(\frac{\ell(2n-1)}{2n^2}\right)\dim_P(S).\]

\end{lem}
\begin{proof}
We use the equivalence of modified box-counting dimension and packing dimension, Theorem \ref{dimequiv}, and the remarks following. Write $A$ as an ascending union $A=\bigcup_{i\in\N} A_i$. We want to show that $\lim_i \overline{\dim}_B(A_i)\geq\dim_P(S)\frac{\ell(2n-1)}{2n^2}$. Let $S_i$ be the set of centers of cubes with vertices whose projections are in $A_i$. By the box-counting bound in the previous lemma, $\overline{\dim}_B(A_i)\geq \frac{\ell(2n-1)}{2n^2}\overline{\dim}_B{S_i}$, whence

\[\lim_i \overline{\dim}_B(A_i)\geq \lim_i \overline{\dim}_B(S_i)\frac{\ell(2n-1)}{2n^2}.\]

Note that $S$ is the ascending union $\bigcup_{i\in\N} S_i$. So, $\lim_i \overline{\dim}_B(S_i)\geq\dim_P(S)$. This completes the proof.
\end{proof}

To get a packing dimension for sets $B$ contaning skeleta of cubes around points in some set $S$, we can take rational translates of the set $B$ to ensure it is in a product form, then apply the above continuous analogue of the the $(n,\ell)$-dimensional lemma to the factors. 

\begin{thm}[Theorem \ref{int:pbbounds} part 1 in the Introduction]
If $B,S\subseteq \R^n$ and $B$ contains the $k$-skeleton of an $n$-cube around every point in $S$, then
\[\dim_P(B)\geq \dim_P (S)\frac{(2n-1)(n-k)}{2n^2} +k.\]
\end{thm}
\begin{proof}
First, let $B'=B+\Q^n$. It is clear (from Theorem $3.1$, for instance) that packing dimension is countably stable. Since $B'$ is a countable union of translates of $B$, it then follows that $\dim_P(B')=\dim_P(B)$. 

If $P$ is a $k$-face in $B$, then there is a  unique $I\in\setc{n}{n-k}$ such that $P_I$ is a singleton, and there is a $k$-plane in $B'$ containing $P$. So, for some collection of sets $A_{\pi}\subseteq \R^{n-k}$, where $\pi\in \Sigma_n$,

\[B'\supseteq\bigcup_{\pi\in \Sigma_n} \pi\left( A_{\pi}\times\R^{k}\right),\]
where every $P_I$ is contained in some $A_\pi$ when $P$ is a $k$-face of a cube in $B$ and $P_I$ is a singleton, and where $\Sigma_n$ acts on $\R^n$ by $\pi(x_1,...,x_n)=(x_{\pi(1)},...,x_{\pi(n)})$. By the product rule, Theorem \ref{dimprod} part (2), and the fact that $\dim_H(\R)=\dim_P(\R)=1$, we have

%specify what part once we fix it

\begin{equation}\label{b'bound}\dim_P(B')\geq \max_{\pi\in \Sigma_n} \dim_P (A_{\pi})+k=\dim_P\left(\bigcup_{\pi\in \Sigma_n} A_{\pi}\right)+k.\end{equation}

Note that, if $x\in S$, there is some $r\in\R^+$ such that there is a $k$-face of a cube in $B$ at distance $r$ from $x$ in every direction. That is, for $I\in\setc{n}{n-k}$, $x_I$ is the center of a cube in $\bigcup_{\pi\in \Sigma_n}A_{\pi}$. This means $A=\bigcup_{\pi\in \Sigma_n} A_{\pi}$ and $S$ satisfy the conditions of the analytic analogue $(n,n-k)$-dimensional lemma, Lemma \ref{nldim}. We then have 
\[\dim_P(A)\geq \frac{(n-k)(2n-1)}{2n^2}\dim_P(S).\] 

Combining this with (\ref{b'bound}), we get
\[\dim_P(B)=\dim_P(B')\geq \dim_P(A)+k\geq \frac{(n-k)(2n-1)}{2n^2}\dim_P(S)+k.\]

\end{proof}

\section{Packing and Vertex Constructions}\label{sec:packcon}

The constructions for packing dimension and the vertex case of box-counting dimension are completely analogous to those in \cite{op}. The key is a lemma generalizing the construction of the Cantor set.

\begin{lem}[ \cite{op} Lemma 6.1 ] 
Let $\{Q_i\}_{i\in\N}$ be a sequence of finite sets in $\R$ such that $\mathrm{diam}Q_i\leq d_i$, $Q_i$ is $\delta_i$-separated (if $x,y$ are in $Q_i$ and distinct, then $|x-y|>\delta_i$), $|Q_i|=\ell_i$, $\sum_{i\in\N} \min Q_i>-\infty$ and $\sum_{i\in\N} \max Q_i<\infty$. Let
\[P=\sum_{i\in\N} Q_i=\left\{\sum_{i\in\N}q_i:q_i\in Q_i \right\}.\]

\begin{enumerate}
\item If there is some $c<1$ such that $d_i\le cd_{i-1}$ for every $i\in \N$, then
\[\overline{\dim}_B P\leq \limsup_{j\rightarrow\infty} \frac{\log(\ell_1,...,\ell_j)}{-\log(d_j)} .\]
\item If $d_i+\delta_i\leq \delta_{i-1}$ then
\[\dim_H(P)\geq \liminf_{j\rightarrow\infty} \frac{\log(\ell_1,...,\ell_j)}{-\log(d_{j+1}\ell_{j+1})}.\]
\end{enumerate}
\end{lem}

%give number, etc

\begin{thm}[Vertex Constructions] \label{vtxcon}
For any positive integer $n$ and any $t\in [0,1]$ there are compact sets $A,T\subseteq \R$ such that 
\begin{equation}\label{vtxdim} \dim_H(T)=\dim_B(T)=\dim_P(T)=t\end{equation}
and for all $x_1,x_2,...,x_n\in T$, there is an $r\in\R^+$ such that
\begin{equation}\label{vtxcond}x_1\pm r, x_2\pm r,...,x_n\pm r\in A\end{equation}
and
\begin{equation}\label{vtxcenterdim} \dim_P(A)=\dim_B(A)=\dim_H(A)=\dfrac{2n-1}{2n}t.\end{equation}
\end{thm}
\begin{proof}
Let $\beta_i=((i-1)!)^{-\frac{2n^2}{t}}$, $A_i=\frac{\beta_i}{i^{2n}}D_{i,n}$, and $T_i=\frac{\beta_i}{i^{2n}}\di{1,i^{2n}-1}$, where $D_{i,n}$ is as in Lemma \ref{digit}. Note that $A=\sum_{i\in \N} A_i$ and $T=\sum_{i\in\N} T_i$ satisfy (\ref{vtxcond}). It remains to verify the dimension conditions (\ref{vtxdim}) and (\ref{vtxcenterdim}). 

One inequality follows from Lemma \ref{nlbox} (for box-counting dimension) and Lemma \ref{nldim} (for packing dimension) with $k=0, \ell=n$ and $S=T^n$, and the product rule, Theorem \ref{dimprod}. The other inequality follows from the dimension inequalities (Theorem \ref{dimineq}) and the fact that the $A_i$ and $T_i$ satisfy the hypotheses of the previous lemma with
$\delta_i=\frac{\beta_i}{i^{2n}},\; d_i=\frac{\beta_i}{i^{2n}}(i^{2n}-1) ,\; \ell_i= i^{2n}$ for the $T_i$ and $\delta=\frac{\beta_i}{i^{2n}},\; d_i=3\frac{\beta_i}{i^{2n}}(i^{2n}-1) ,\; \ell_i= |D_k|\leq O(i^{2n-1})$ for the $A_i$ (one can check this).

\end{proof}

We prove a slight generalization of of Theorem \ref{int:pbcon} Part 1 from the introduction. This will mostly be subsumed by the box-counting construction, except in the case where $s=n$.

\begin{thm}[Packing and Weak Box-counting Construction] \label{pbcon}
For every $n,k$ with $0 \leq k< n$ and every $s\in [0,n]$, there are compact sets $B, S\subseteq \R^n$ where $B$ contains the $k$-skeleton of an $n$-cube around every point in $S$, and
\[\dim(S)=s\quad\mbox{and}\quad \dim(B)= k+ \dim(S)\dfrac{(n-k)(2n-1)}{2n^2}\]
where $\dim$ is either packing or box-counting dimension.
\end{thm}
\begin{proof}
First, we want an $S$ of the appropriate dimension and an $A$ with
\begin{equation}\label{packconstcond}\forall x_1,...,x_n\in S\; \exists r\in\R^+ \;\forall I\in \setc{n}{n-k} \;\forall i\in I :\quad x_i\pm r\in A.\end{equation}
We will get $B$ by taking a power of $A$ and interleaving copies of $\R$. The condition (\ref{packconstcond}) is clearly satisfied if there is some $r$ such that $x_i\pm r\in A$ for every $1\leq i\leq n$.

Apply the vertex construction, Lemma \ref{vtxcon}, to get a compact $T\subseteq [0,1]$ of dimension $\dfrac{s}{n}$ and compact $A\subseteq \R$ of dimension $\frac{2n-1}{2n^2}s$ satisfying
\[\forall x_1,...,x_n\in T \;\exists r\in\R^+ \;\forall 1\leq i \leq n: \quad x_i\pm r\in A\]
  Take $S=T^n$ and $B=\bigcup_{I\in \setc{n}{n-k}} \prod_{i=1}^n A_{I,i}$ where

 \[A_{I,i}=\left\{\begin{array}{ll} A &  \mathrm{if } i\in I \\ \R & \textrm{otherwise}\end{array}\right..\] 
Then,
  \[\dim(S)=\dim(T)n=s\]
and
\[\quad \dim(B)=k+\dim(A)(n-k)=k+\dfrac{(n-k)(2n-1)}{2n^2}s.\]

For any $x\in S$, there is some $r\in\R^+$ such that, for any $I\in \setc{n}{n-k}$ and $\sigma\in \{-1,1\}^n$, $x_I+r\sigma_I\in A^k$, so $x+r\sigma\in \prod_{i=1}^n A_{I,i}\subseteq B$. So, $B$  contains the $k$-skeleton of an $n$-cube around each point in $S$.
\end{proof}

%DO NOT NEED THE BOUND. PRODUCT RULE IS ENOUGH.

Note that when $s=n$ the bounds for box-counting and packing dimension are equal, and the $S$ constructed above is exactly the unit cube.

\section{Box-counting Constructions}\label{sec:boxcon}

The box-counting construction is again completely analogous to the construction in \cite[Section $6.2$]{op}, however there are many more details to keep track of. For completeness we provide the entire argument below:

\begin{thm}[Theorem \ref{int:pbcon} part 2 in the Introduction] \label{boxcon}
For every $n,k$ with $0\leq k< n$ and every $s\in [0,n]$, there are compact sets $B, S\subseteq \R^n$ where $B$ contains the $k$-skeleton of an $n$-cube around every point in $S$, and
\[\underline{\dim}_B(S)=s\quad\mbox{and}\quad \overline{\dim}_B(B)= \max\left\{k,\left(1-\frac{(n-k)}{2n^2}\right)s\right\}\]

\end{thm}

The set $A$ will be stiched together from analogues of the $D_{i,n}$ in Lemma 2.1. We first need to get better control of how difficult our $D_{i,n}$ analogues are to cover with intervals of a given length.

\begin{lem} For any $n$, there is a sequence of sets $\{A_N\}_{N\in\N}$ of natural numbers such that the following hold:
\begin{enumerate}
\item For every $N$ and every $x_1,...,x_n \in \di{1,N}$ there is $r\in \di{1,3N}$ such that for all $1\leq i\leq n, \;\;x_i\pm r\in A_N$.
\item For every $\delta>0$ there is $C=C(\delta)\in\R^+$ such that for every $R\in[1,N]$, the set $A_N$ can be covered by $CN^\delta(N/R)^{(2n-1)/(2n)}$ intervals of length $R$.
\end{enumerate}
\end{lem}
\begin{proof}[Proof of Lemma]
We first consider $N$ of the form $(p!)^{2n}$ for some $p$. Here let
\[A_N=(p!)^{2n}\sum_{i=1}^p \frac{(D_{i,n}\cup (D_{i,n}-1))}{(i!)^{2n}}\]
where $D_{i,n}$ is as in Lemma \ref{digit}. Suppose $x_1,..., x_n\in\di{1,N}$. Write $x$ as
\[x_k=N \sum_{i=1}^p \frac{x_{k,i}}{(i!)^{2n}}\]
where $x_{k,i}\in \di{0,i^{2n}-1}$ and at least one $x_n$ is not $0$. For $1\leq i\leq p$, let $r_i\in \di{1,i^{2n}}$ be such that $x_{k,i}\pm r_i\in D_{i,n}$ for each $k$ and let
\[r=N\sum_{i=1}^p \frac{r_i}{(i!)^{2n}}.\]
Then, $x\pm r\in A_N$ and $r\leq N\sum_{i=1}^p \frac{1}{(i-1)!}\leq 3N.$

We now need to verify the covering property $(2)$. For $1\leq j\leq p$, $A_N$ can be covered in
\[2^n|D_{2,n}||D_{3,n}|...|D_{j+1,n}|=O(|D_{1,n}|...|D_{j,n}|)=O(1)^j(j!)^{2n-1}\]

intervals of length

\[(p!)^{2n}\sum_{m=j+1}^p \frac{(3m)^{2n}}{(m!)^{2n}}\leq \frac{3^{2n+1}(p!)^{2n}}{(j!)^{2n}}.\]

Define $R_j$ as $\frac{3^{2n+1}(p!)^{2n}}{(j!)^{2n}}$. We have just shown that $A_N$ can be covered by $O(1)^j(N/R_j)^{(2n-1)/(2n)}$ intervals of length $R_j$ (independent of $\delta$). Note that the $R_j$ are increasing and $R_p\geq 3^{2n+1}$. For general $R$ we interpolate as follows. By making $C$ large enough, we may assume $R\in[R_p=3^{2n+1}, N]$. Select $j$ such that $R_{j+1}<R<R_j$; $A_N$ can be covered by $O(N/R_{j+1}^{(2n-1)/(2n)})$ intervals of length $R$. Since $\frac{\log(j+1)!}{\log(j)!}\rightarrow 1$ as $j\rightarrow\infty$, for every $\delta>0$, there is some $C>0$ such that $R_j<CR_{j+1}^{1+\delta}\leq O(N^\delta)R_{j+1}$. This tells us $A_N$ can be covered by $O(N^{2\delta})(N/R)^{(2n-1)/(2n)}$ intervals of length $R$. 

For general $N$ we again interpolate between values of $(p!)^{2n}$ by using the fact that 
\[\lim_{p\rightarrow\infty} \frac{\log((p+1)!)}{\log(p!)}=1 .\]\end{proof}
\begin{proof}[Proof of Theorem \ref{boxcon}]
The case where $s=n$ is handled by Theorem \ref{pbcon}, so we may assume $s\not=n$. Let $\alpha>0$ be such that $s=\frac{n\alpha}{1+\alpha}$. For each $i\in \N$, let $N_i=\lfloor 2^{\alpha i}\rfloor$ and define
\[S_i=\di{1,N_i-1}^n,\quad B_i=\bigcup_{I\in \setc{n}{n-k}} \prod_{1\leq i\leq n} A'_i\]
where $A'_i=\left\{ \begin{array}{cc} A_{N_i} & i\in I \\ \;[-3N_i, 4N_i] & \textrm{otherwise} \end{array}\right.$ and $A_{N_i}$ is as in the above lemma. Let $\varepsilon_i=2^{-(1+\alpha)i}$ and define\bigskip

\[
S=\{(0,...,0)\}\cup \bigcup_{i\in \N} \left((2^{-i},0,...,0)+\varepsilon_iS_i\right)
\] \[
B=C_0\cup\{(0,...,0)\}\cup\bigcup_{i\in\N} \left( (2^{-i},0,...,0)+\varepsilon_iB_i\right),
\]where $C_0$ is the $k$-skeleton of a cube of unit side length centered at the origin. 

The set $S$ then consists of a sequence of translated shrinking copies of the discrete cube and $B$ contains the $k$-skeleton of a cube around each of these points, since $B_i$ contains the $k$-skeleton of an $n$-cube around each point in $S_i$. We then have that $B$ and $S$ satisfy $(1)$ in the statement of the theorem. It remains to show that these sets have the correct dimension.

We first verify that $\underline{\dim}_B(S)\geq s.$ $S$ contains a translate of $\varepsilon_iS_i$, so contains $|N_i|^n=\Omega(2^{n\alpha i})=\Omega(\varepsilon_i^{-s})$ points at pairwise distance at least $\varepsilon_i$. Interpolating an arbitrary $\varepsilon\in (0,1)$ between consecutive values of $\varepsilon_i$, we deduce that $\underline{\dim}_B(S)\ge s$.

We now show that $\overline{\dim}_B(B)\le \max\{k,(1-\frac{n-k}{2n^2})s\}+O(\delta)$ for any $\delta>0$. To get the desired estimate, we will count the number of cubes of side length $\varepsilon_i$ needed to cover $B$. First fix $i$ and and decompose $B$ as $B'_i\cup B''_i\cup B'''_i$, where
\[
B'_i=C_0\cup \bigcup_{j=1}^{i-1}\left((2^{-j},0,...,0)+\varepsilon_j B_j\right)
\]\[
B_i''=\bigcup_{j:\varepsilon_j\leq \varepsilon_i\leq N_j\varepsilon_j} \left((2^{-j},0,...,0)+\varepsilon_j B_j\right)
\]\[
B'''_i=\bigcup_{j:N_j\varepsilon_j<\varepsilon_i} \left((2^{-j},0,...,0)+\varepsilon_j B_j\right).
\]

We first count the $\varepsilon_i$ balls needed to cover $B'_i$. Note that for $j<i$, $\varepsilon_j B_j$ consists of discrete ${n\choose k}|A_{N_j}|^{n-k}$ $k$-cubes of side length $O(\varepsilon_jN_j)$. Since, by part $(2)$ of the above lemma applied to $R=1$, $|A_{N_j}|=O\left(2^{\alpha j(\frac{2n-1}{2n}+\delta)}\right)$, $\varepsilon_jB_j$ can be covered by 

\[
O(1)|A_{N_j}|^{n-k}\left(\frac{\varepsilon_jN_j}{\varepsilon_i}\right)^k=O(1)2^{j\alpha(\frac{2n-1}{2n}+\delta)(n-k)}2^{-jk}\varepsilon_i^{-k}=
\]\[
O(1)\varepsilon_i^{-k}2^{(\frac{(2n-1)(n-k)}{2n}\alpha-k)j+\delta(n-k)\alpha j}
\]balls of radius $\varepsilon_i$.  Straightforward calculation shows
\[\left(\frac{(2n-1)(n-k)}{2n}\right)\alpha-k\leq 0\]
exactly when
\[s\le \frac{2n^2k}{2n^2-(n-k)}=\frac{k}{1-\frac{(n-k)}{2n^2}},\]
and in this case (shrinking $\delta$ by a factor depending on $s$), $B'_k$ can be covered by $O(\varepsilon_i^{-k})$ balls of radius $\varepsilon_i$ (this means, more or less, that $k$ is the minimum dimension for $B$.). Otherwise, $s> \frac{2n^2}{2n^2-(n-k)}$ and $B'_i$ can be covered by 

\[
\sum_{j=1}^{i-1}O(1)\varepsilon_i^{-k}2^{(\frac{(2n-1)(n-k)}{2n}\alpha-k)j+\delta(n-k)\alpha j}\leq
\]\[
O(1)i\varepsilon_i^{-k}2^{(\frac{(2n-1)(n-k)}{2n}\alpha-k)i+\delta(n-k)\alpha }=
\]\[
O(1)i2^{(\frac{(2n-1)(n-k)}{2n}+k)\alpha i}2^{\delta(n-k)\alpha i}=O\left(\varepsilon_i^{-(\frac{(2n-1)(n-k)}{2n^2}+\frac{k}{n})s-O(\delta)}\right)
\]

We next count the $\varepsilon_i$ balls needed to cover $B''_i$. Suppose $\varepsilon_j<\varepsilon_i<\varepsilon_jN_j$. Again using the second part of the previous lemma, $A_{N_j}$ can be covered by \[O(N^\delta_j)(N_j/R)^{(2n-1)/(2n)}\] intervals of length $R$ for any $R$ in $[1,N_j]$. So, $B_j$ can be covered by \[O(N^{\delta}_j)(N_j/R)^{\frac{(2n-1)(n-k)}{2n}+k}\] balls of radius $R$, and $\varepsilon_j B_j$ can be covered by the same number of balls of radius $\varepsilon_j R$. Applying this to $R=\varepsilon_i/\varepsilon_j$ (which is in $[1, N_j]$ by assumption), we get $\varepsilon_jB_j$ can be covered by

\[O(N^{\delta}_j)\left(\frac{N_j\varepsilon_j}{\varepsilon_i}\right)^{\frac{(2n-1)(n-k)}{2n}+k}=O(1)\left(2^{(1-O(\delta))j}\varepsilon_i\right)^{-(\frac{(2n-1)(n-k)}{2n}+k)}\]

balls of radius $\varepsilon_i$. Taking $\delta$ small enough and summing over $j\geq i$ (via the geometric sum formula,) we can cover $B''_i$ by

\[O(1)\left(2^{(1-O(\delta))i}\varepsilon_i\right)^{\frac{(2n-1)(n-k)}{2n}+k}=O\left(\varepsilon_i^{-s(\frac{(2n-1)(n-k)}{2n^2}+\frac{k}{n})-O(\delta)}\right)\] 
balls of radius $\varepsilon_i$

Finally, we count the number of balls needed to cover $B'''_i$. The smallest $j$ such that $\varepsilon_jN_j<\varepsilon_i$ satisfies $2^{-j}\leq O(\varepsilon_i)$, so $B'''_i$ has diameter $O(\varepsilon_i)$, and can be covered by $O(1)$ balls of radius $\varepsilon_i$.

Putting the above estimates together, we get that $B$ can be covered by $O(\varepsilon_i^{-k})$ balls of radius $\varepsilon_i$ when $s(1-\frac{n-k}{2n^2})<k$ and by $O(1)\left(\varepsilon_i^{(1-\frac{n-k}{2n^2})s+O(\delta)}\right)$ balls of radius $\varepsilon_i$ otherwise. This implies that $\overline{\dim}_B(B)\leq \max\{k,(1-\frac{n-k}{2n^2})s\}+O(\delta)$ for small enough $\delta$. Letting $\delta$ tend to $0$ gives the desired result.
\end{proof}
\section{Hausdorff Results}\label{sec:haus}

The results for Hausdorff dimension are mostly trivial generalizations of the results in \cite{op}. We will formulate a conjecture based on these partial results. This conjectures was recently settled in the affirmative by Marianna Cs\"ornyei and Tam\'as Keleti.

\begin{thm}[Theorem \ref{int:hausbound} in the Introduction]\label{hausbound}
If $B,S\subseteq\R^n$ and $B$ contains the $k$-skeleton of a cube around every point in $S$, then $\dim_H(B)\geq \max\{\dim_H(S)-1,k\}.$ 
\end{thm}
\begin{proof}
Let $P$ be the projection onto the plane normal to $(1,...,1)$. Since $B$ contains the vertices of a cube around each point in $S$, $P(S)\subseteq P(B)$. So, \[\dim_H(S)\leq \dim_h(P(S))+1\leq \dim_H(P(B))+1\leq \dim_H(B)+1.\]
\end{proof}

This bound seems quite weak, but it is sharp in the case when $n=2$, as shown in \cite{op}. We can extend this a bit further.

\begin{thm}
There is a set $B\subseteq\R^n$ which contains the boundary, i.e. the $(n-1)$-skeleton, of an $n$-cube around each point in $\R^n$ such that $\dim_H(B)=k$.
\end{thm}
\begin{proof}
Let $A\subseteq\R^\ell$ be a comeager set of Hausdorff dimension $0$. Since $\bigcap_{1\leq i\leq n} f_i(A)$ is comeager for any finite family of affine maps $\{f_i\}$, we have, for any $x_1,...,x_n\in\R$
\[(A-x_1)\cap (x_1-A) \cap (A-x_2)\cap (x_2-A)\cap...\cap (A-x_n)\cap (x_n-A)\not\subseteq \{0\}.\]

This means there is some $r\in\R^+$ such that $x_i\pm r\in A$ for all $1\leq i\leq n$. It then follows that $B=\bigcup_{1\leq i\leq n} \R^{i-1}\times A \times\R^{n-i}$ is as desired.
\end{proof}

Using the methods in \cite{op} and in Section \ref{sec:ortho}, one can make compact sets $B$ containing the boundary of a skeleton of a cube around each point in $[0,1]$ of hausdorff dimension $n-1$.

\begin{cor}[Theorem \ref{int:hauscon} in the Introduction]
For $0\leq k<n$, $s\in [n-k,n]$, there are $G_\delta$ $B,S\subseteq \R^n$ where $B$ contains the $k$-skeleton of an $n$-cube around each point in $S$, $\dim_H(B)=\max\{k,s-1\}$ and $\dim_H(S)=s$.

Further, if $k=0$, there are $B,S$ as above for $s\in[n-1,n]$.
\end{cor}
\begin{proof}
The case where $n-1=k\not=0$ is implied by the previous theorem. The case where $n=2, k=0$ follows from \cite[Theorem 3.5]{op}. We can extend these results to constructions in higher dimensions simply by taking products with $\R$. Fix $n$. The constructions for each $k$ are nearly identical, but it is clearer to separate the case where $k=0$. 

Suppose $k=0$. For any $s\in [n-1,n]$, there are sets $S', B'\subseteq\R^2$ such that $\dim_H(S')=s-(n-2)$, $\dim_H(B')=s-(n-1)$ and $B'$ contains the vertices of a square around every point in $S'$. (Such sets are guaranteed to exist by the case when $n=2$). Take $B=B'\times \R^{n-2}$ and $S=S'\times \R^{n-2}$. Clearly $B$ contains the vertices of an $n$-cube around each point in $S$. By the product rules, Theorem 3.3, $\dim_H(S)=\dim_H(S')+(n-2)=s$ and $\dim_H(B)=\dim_H(B')+(n-2)=s-1$.

Suppose $k>0$. For any $s\in[n-k,n]$, there are $S', B'\subseteq \R^{k+1}$ such that $\dim_H(S')=s-(n-k-1)$, $\dim_H(B')=s-(n-k)$, and $B'$ contains the $k$-skeleton of a cube around each point in $S$ (these are guaranteed to exist by the case where $n=k+1$), and take $S=S'\times \R^{n-k-1}$ and $B=B'\times\R^{n-k-1}.$
\end{proof}

On the basis of these results, we conjecture that the bound $\dim_H(B)\geq \dim_H(S)-1$ is sharp in all cases. This was proven by Cs\"ornyei and Keleti using genericity arguments \cite{CK}.

\section{Orthoplex Vertex Problem}\label{sec:ortho}

It is fairly natrual to consider the vertex problem for the dual polytope of the cube, the orthoplex. That is how small can $B$ be and still contain the vertices of an orthoplex around each point in a set $S$ of a given size. We can find sharp bounds for each notion of dimension discussed above.

(The simplex does not have as clear an additive structure as the orthoplex or cube, and the faces of an orthoplex are not in general orthoplexes. So, the related questions for simplices and skeleta of orthoplexes are not quite as natural. Of course, these may turn out to have clean and interesting answers.)

\begin{thm}[Theorem \ref{int:orthobound} in the Introduction]
Let $B,S\subseteq\R^n$ such that for all $x\in S$ there is some $r\in\R^+$ such that $x\pm re_i\in B$, where $e_i$ is the $i^{th}$ standard basis vector. Then,

\begin{enumerate}
\item $\dim_H(B)\geq\dim_H(S)-1$,
\item If $B,S$ are finite, then $|B|\geq\Omega(|S|^{\frac{2n-1}{2n}})$,
\item $\underline{\dim}_B(B)\geq \frac{2n-1}{2n}\underline{\dim}_B(S)$ and $\overline{\dim}_B(B)\geq \frac{2n-1}{2n}\overline{\dim}_B(S)$, and
\item $\dim_P(B)\geq\frac{2n-1}{2n}\dim_P(S)$.
\end{enumerate}
\end{thm}
\begin{proof}
Precisely the same argument as for the Hausdorff bound of the cube problem (Theorem \ref{hausbound}) gives $(1)$. The $n$-dimensional lemma, Lemma \ref{nlem} gives $(2)$. And, $(2)$ gives $(3)$ and $(4)$ exactly as Theorem \ref{nllem} gives Lemma \ref{boxbound} and Corollary \ref{nldim}.  
\end{proof}

We can show that all of these bounds are sharp.

\begin{thm}[Theorem \ref{int:orthocon} in the Introduction] \label{orthocon}
For any $n, p\in \N$ and $s\in[0,n]$, we can find sets $B_H, B_B, B_P, B_f, S_H, S_B, S_P,$ and $S_f\subseteq \R^n$ such that $B_X$ contains the vertices of an orthoplex around each point in $S_X$ for any $X\in\{H,B,P, f,\}$, $\dim_X(S_X)=s$ for $X\in\{H,B,P\}$, $|S_f|=p$, and
\begin{enumerate}
\item $\dim_P(B_P)\leq \frac{2n-1}{2n}\dim_P(S_P)$
\item $\overline{\dim}_B(B_B)\leq \frac{2n-1}{2n} \underline{\dim}_B(S_B)$
\item $|B_f|\leq O(|S_f|^{2n-1}{2n})$
\item $\dim_H(B_H)\leq \dim_H(S_H)-1$.
\end{enumerate}

\end{thm}
\begin{proof}[Proof for discrete, box-counting, and packing dimension.] A cube has vertices in at least $n$ independent directions from the center, so we may take the vertex constructions for the cube and apply a linear map to get constructsion for (2),(3), and (4). More explicitly, there are $B'_X$ and $S'_X$ for $x\in \{P,B,f\}$ satisfying the conditions $(1), (2), $ and $(3)$ where $B'_X$ contains the vertices of a cuber around each point in $S'_X$ (Theorems 5.4, 6.1, and 2.3 respectively). There is a basis $b_1,...,b_n$ for $\R^n$ such that $x\pm b_i\in B'_X$ for every $x\in S'_X$ and $1\leq i\leq n$. Let $g$ be the invertible linear function defined by $g(b_i)=e_i$ Then, $S_X=g(S'_x)$ and $B_X=g(B'_X)$ is as desired.
\end{proof}

To get a construction for the Hausdorff dimension, we will generalize the arguments \cite[Section 3]{op}. The important points are the splicing operation $SPL$, which takes a sequence of sets in some $\R^m$ and returns a set in $\R^m$, and a method for constructing sets with strong intersection properties. 

\begin{lem} \label{spllem}
Let $X_i, Z_i, Y_i\subseteq \R^m$ for some $m$ and all $i\in\N$.

\begin{enumerate}
\item $\min_{n\in\N}\dim_H((Z_i))\leq\dim_H(\dim_H(SPL((Z_i)_{i\in\N})\leq\min_{i\in\N}\overline{\dim}_B(Z_i)$, provided there are only finitely many distinct $Z_i$ and each occurs infinitely often, and
\item $SPL((X_i\times Y_i)_{i\in\N})=SPL((X_i)_{i\in\N})\times SPL((Y_i)_{i\in\N})$.

\end{enumerate}

\end{lem}

These properties of $SPL$ are established in \cite[Lemma 3.7]{op} and the remarks preceding it.

\begin{lem}
For any compact set of affine maps on $\R$, $F$, there are $K_p\subseteq \R$ and affine $g_p$, $0<p\leq M$, such that
\begin{enumerate}
\item $A:= \bigcup_{p=1}^M g_p(K_p)$ satisfies $\bigcap_{f\in F'} f(A)\not\subseteq \{0\}$ for any finite $F'\subseteq F$, and
\item $K_p\subseteq\bigcup_{k\in\Z} (SPL((X^{(p)}_i))+k)$ where

\[X_i^{(p)}=\left\{\begin{array}{cc} \{0\} & \textrm{if }\exists j:\;i=(2j-1)2^{p} \\ \;[0,1] & \textrm{otherwise}\end{array}\right..\]
\end{enumerate}
\end{lem}

The above Lemma is implied by the proofs of Propositions 3.1 and 3.6 in \cite{op}.

\begin{proof}[Proof of \ref{orthocon} for Hausdorff dimension.]
We may assume $s\in [1,n]$. Let $B$ be any compact subset of $[0,1]$ such that $\dim_H(B)=\dim_B(B)=\frac{s-1}{n-1}$. Let $F$ be the set of functions 
\[F=\{x\mapsto y-x:y\in B\}\cup\{x\mapsto x-y:y\in B\}\]
and let $A$ be the corresponding set given by part 1 of the previous lemma. For $0<q\leq n$, let 

\[C^{(q)}_i=\left\{\begin{array}{cc} I & \exists j,p: 2^pj=i\; and\; j=2q-1\;(mod\;2n)\\ B & otherwise  \end{array}\right..\] 

Let $C^{(q)}=SPL((C^{(q)}_i)_{i\in\N})$ and define

\[
S_H=\prod_{1\leq p\leq n}C^{(p)} 
\] 
and
\[
B_H=\bigcup_{\pi\in \Sigma_n}\bigcup_{i\in\di{1,n}} \pi\left(A\times\prod_{q\not=i} C^{(q)}\right)
.\] Then, since for any $x_1,...,x_n\in \bigcup_{1\leq q\leq n} C^{(q)}$
\[(A-x_1)\cap (x_1-A) \cap (A-x_2)\cap (x_2-A)\cap...\cap (A-x_n)\cap (x_n-A)\not\subset \{0\},\]
if $x\in S_H$, $1\leq i\leq n$, there is some $r\in\R^+$ such that $x\pm re_i\in B_H,$ as desired. (Note that it follows that $\dim_H(B_H)\geq \dim_H(S_H)-1.$)

We have that $\prod_{1\leq p\leq n}C_i^{(p)}$ is (up to a permutation) $B^m\times I^{n-m}$ for some $m<n$, and $B^{n-1}\times I$ occurs infinitely often. So, by the properties of $SPL$, Lemma \ref{spllem}, 

\[\dim_H(S_H)=(n-1)\dim_H(B)+1=s.\]

We also have that, for any $I\in\setc{n}{n-1}$, $X^{(p)}_i\times \prod_{q\in I} C^{(q)}_i$ is (up to a permutation) one of $\{0\}\times B^{n-1}$, $\{0\}\times I\times B^{n-2}$, $I\times B^{n-1}$, or $I^2\times B^{n-2}$, and each of these shows up infinitely often. So $\dim_H(SPL((X^{(p)}_i))\times \prod_{q\in I} C^{(q)})=(n-1)\dim_H(B)=s-1$. Finally,
\begin{align*}\dim_H(B_H)&=\max_{1\leq i\leq n} \dim_H(A\times\prod_{p\not=i} C^{(p)})\\
\;& =\max_{1\leq i\leq n}\max_{1\leq j\leq M} \dim_H(K^{(j)}\times\prod_{p\not=i} C^{(p)}) \\
\; &\leq \max_{1\leq i\leq n}\max_{1\leq j\leq M} \dim_H(SPL((X^{(j)}_\ell))\times\prod_{p\not=i} C^{(p)})\\
\; &= s-1  
\end{align*}

\end{proof}


\begin{thebibliography}{99}

\bibitem{CK}

M. Cs\"ornyei and T. Keleti, forthcoming.

\bibitem{Bourgain}

J. Bourgain, Averages in the plane over convex curves and maximal operators, \emph{J. Analyse
Math.}, \textbf{47} (1986), 69-85.


\bibitem{packing}

 K. Falconer, \emph{Fractal geometry, 2nd edition}, John Wiley \& Sons, Inc. (Hoboken, NJ, 2003).


\bibitem{kklproof}

P. Frankl, A new short proof for the Katona--Kruskal theorem, \emph{Discrete Mathematics}, \textbf{48} (1984), 327-329.

\bibitem{Ka}

G. Katona, A theorem of finite sets, in: \emph{Theory of Graphs}, Erd\"os P., Katona G. (eds.), Akad\'emia Kiad\'o (Budapest, 1968), 187-207.

\bibitem{op}

T. Keleti, D.T. Nagy and  P. Shmerkin, Squares and Their Centers, \emph{J. Anal. Math.}, (to appear).

\bibitem{Kr}

J.B. Kruskal, The number of simplices in a complex, in \emph{Mathematical Optimization Techniques}, University of California Press (Oakland, CA, 1963), 251-278.

\bibitem{L}

L. Lov\'asz, \emph{Combinatorial Problems and Exercises}, North-Holland (Amsterdam, 1979). 


\bibitem{Marstrand}

J. M. Marstrand, Packing circles in the plane, \emph{Proc. London Math. Soc.} \textbf{55} (1987), 37-58.


\bibitem{inequalities}

P. Mattila, \emph{Geometry of sets and measures in Euclidean spaces}, Cambridge University Press (Cambridge, 1995.)

\bibitem{N}

D.T. Nagy, private communication.

\bibitem{sharples}

J. Robinson and N. Sharples, Strict inequality in the box-counting dimension product formulas, \emph{Real Anal. Exchange} \textbf{38} (2012) 95-120.

\bibitem{nballs}

E. Stein, Maximal Functions: Spherical Means, \emph{Proc. Natl. Acad. Sci. USA}
\textbf{73} (1976), 2174-2175.

\bibitem{dimprodref}

C. Tricot, Two definitions of fractional dimension, \emph{Math. Proc. Camb.
Phil. Soc.} \textbf{91} (1982), 57-74.

\end{thebibliography}
\end{document}